\documentclass[a4paper,11pt,pdftex,dvipsnames]{amsart}
\usepackage[margin=1.2in]{geometry}

\usepackage{amsmath,amssymb,amsfonts,amsthm}
\usepackage{eqnarray}
\usepackage{tikz, tikz-cd}
\usetikzlibrary{positioning}
\usepackage{caption, subcaption} 
\usepackage{enumerate}
\usepackage{geometry}
\usepackage{subcaption}
\usepackage{stackengine}
\usepackage{enumitem} 
\usepackage{adjustbox}
\usepackage{mathbbol}
\usepackage{verbatim}
\usepackage{xcolor}
\definecolor{darkgreen}{rgb}{0,0.50,0} 
\definecolor{darkred}{rgb}{0.55,0,0}
\definecolor{darkblue}{rgb}{0,0,0.6}
\definecolor{darkteal}{rgb}{0.0, 0.25, 0.5}
\usepackage[pdfborder=0,pagebackref,colorlinks,citecolor=darkgreen,linkcolor=darkteal,urlcolor=darkblue]{hyperref}
\usepackage{wrapfig
}

\usepackage{thmtools}
\usepackage{thm-restate}
\usepackage{todonotes}
\newcounter{todocounter}

\usepackage[capitalise]{cleveref}

\newcommand{\newrefformat}[2]{}

\crefname{lemma}{Lemma}{Lemmas}
\crefname{theorem}{Theorem}{Theorems}
\crefname{definition}{Definition}{Definitions}
\crefname{proposition}{Proposition}{Propositions}
\crefname{remark}{Remark}{Remarks}
\crefname{corollary}{Corollary}{Corollaries}
\crefname{equation}{Equation}{Equations}
\crefname{ex}{Example}{Examples}
\crefname{appsec}{Appendix}{Appendices}

\newtheorem{lemma}{Lemma}
\newtheorem{theorem}[lemma]{Theorem}
\newtheorem{corollary}[lemma]{Corollary}
\newtheorem{proposition}[lemma]{Proposition}
\theoremstyle{definition}
\newtheorem{definition}[lemma]{Definition}
\newtheorem{example}[lemma]{Example}

\newtheorem{thmx}{Theorem}

\newcommand{\renee}[1]{{\color{magenta}{#1}}}

\newcommand{\gillout}{\bgroup\markoverwith
	{\textcolor{darkgreen}{\rule[.5ex]{2pt}{0.4pt}}}\ULon}
\newcommand{\reneeout}{\bgroup\markoverwith
	{\textcolor{magenta}{\rule[.6ex]{3pt}{0.6pt}}}\ULon}

\newcommand{\T}{\mathcal{T}}
\newcommand{\BHV}{\mathrm{BHV}}
\newcommand{\Rp}{\mathbb{R}_{\geq 0}}

\newcommand{\dpp}{\mathbb{d}_P}
\newcommand{\dhh}{\mathbb{d}_H}

\begin{document}

\title[Nested tree space: a geometric framework for co-phylogeny]{Nested tree space: \\a geometric framework for co-phylogeny}

\author[G. Grindstaff, R. S. Hoekzema]{Gillian Grindstaff$^{1, 2}$, Renee S. Hoekzema$^{3}$}
\address{$^1$Mathematical Institute, University of Oxford\\
$^2$ Department of Mathematics, UCLA\\
$^3$ Department of Mathematics, Vrije Universiteit Amsterdam}

\email{r.s.hoekzema@vu.nl}

\date{\today}


    \begin{abstract}
Nested (or reconciled) phylogenetic trees model co-evolutionary systems in which one evolutionary history is embedded within another. We introduce a geometric framework for such systems by defining $\sigma$-space, a moduli space of fully nested ultrametric phylogenetic trees with a fixed leaf map.

Generalizing the $\tau$-space of Gavryushkin and Drummond \cite{ultra}, $\sigma$-space is constructed as a cubical complex parametrised by nested ranked tree topologies and inter-event time coordinates of the combined host and parasite speciation events. We characterise admissible orderings via binary {\em nesting sequences} and organise them into a natural poset. We show that $\sigma$-space is contractible and satisfies Gromov’s cube condition, and is therefore CAT(0). In particular, it admits unique geodesics and well-defined Fréchet means. We further describe its geometric structure, including boundary strata corresponding to cospeciation events, and relate it to products of ultrametric tree spaces via natural forgetful maps.
\end{abstract}

\maketitle

\section{Introduction}

Phylogenetic reconciliation studies the common evolution of co-evolving biological entities where one evolution is taking place inside the other, with the aim to model the phylogenetic descent of the pair by ``a tree inside a tree''
\cite{menet2022phylogenetic}. 
Two classical examples are host-parasite systems and genes inside evolving species, but other systems include symbionts inside hosts and species evolving in separating geographical areas \cite{page1998trees}. More recently, similar ideas have been applied in the study of phylosymbiosis, the common evolution of species with their microbiome \cite{phylosymbiosis,Brooks2016, Kohl2020Phylosymbiosis}.
In the most restrictive scenario, the evolution of the associate is strictly bound to the evolution of the host, such that the common evolution can be thought of as a \textit{reconciled} or \textit{nested tree}. 
The current work studies reconciled phylogenetic trees as points in a geodesic space of all possible nested trees, providing a foundation for the development of geometrically-informed statistics.

The co-evolutionary relationship between hosts and their associates— parasites,
commensals, endosymbionts, or microbiomes—is widely studied
in biology. Fahrenholz's rule, formalised in the 20th century, poses that
parasite phylogenies recapitulate host phylogenies, and this motivated large-scale comparative studies supporting the hypothesis in certain systems \cite{brooks1988,hafner1988}.
The general assumption is that a host speciation forces a \emph{cospeciation} in the parasite tree, such that host and parasite lineages diverge simultaneously.
Three evolutionary events can disrupt strict mirroring, however: \emph{duplication}, in
which the parasite speciates independently within a host lineage;
\emph{host-switching} (horizontal transfer), in which a parasite moves to a
phylogenetically separated host; and \emph{sorting events}, in which a parasite
fails to colonise one of the two lineages produced by a host speciation event.
These events and their consequences for
tree topology have been studied extensively since Page~\cite{page1994, Page2003} and the Jungles algorithm of
Charleston~\cite{charleston1998}, with many subsequent algorithmic and
statistical advances~\cite{conow2010,santichaivekin2021}. An analogous problem arises in molecular evolution: gene trees embedded
within species trees can differ topologically due to gene duplication, horizontal
transfer, loss, or incomplete lineage sorting~\cite{maddison1997,PageCharleston1997}. 
In the
gene/species tree context, pairs of genes that diverged via speciation are
called \emph{orthologs}, while pairs arising from gene duplication are called
\emph{paralogs}.

We focus our attention in this work on the regime of only sorting, duplication and cospeciation; that is, horizontal transfer
(host-switching)
events are absent or negligible. This allows us to assume that the parasite tree can be fully
reconciled with the host tree ~\cite{page1994}.
Gene trees within eukaryotic species satisfy this assumption automatically, since
horizontal gene transfer is rare in the absence of bacterial-style conjugation.

The study of phylogenetic trees as points in a metric space was initiated with the introduction of the \emph{BHV space} by Billera, Holmes, and
Vogtmann~\cite{BHV}, and has since been extended to a wide range of phylogenetic tree spaces with varying assumptions and metrics.
For \emph{timed} (ultrametric) phylogenetic trees, i.e. trees in which all leaves are
equidistant from the root, so that internal nodes carry absolute divergence
times, Gavryushkin and
Drummond~\cite{ultra} introduced two natural metric spaces:
the \emph{$t$-space}, parametrised by absolute node times, and the
\emph{$\tau$-space}, parametrised by inter-coalescent intervals $\tau_i =
t_i - t_{i-1}$.  
Both spaces are CAT(0)\footnote{ The non-positive
curvature (CAT(0)) condition implies that any two points are connected by a
unique geodesic, which in turn enables the definition of Fr\'echet means and
other geometric statistics for samples of trees~\cite{OwenProvan2011, miller2015}, including principal component analysis~\cite{Nye2011},
confidence regions~\cite{willis2019}, and kernel density
estimation~\cite{weyenberg2014}.
}, with $\tau$-space additionally
admitting a polynomial-time geodesic algorithm via the Owen--Provan
method~\cite{OwenProvan2011}.  The $\tau$-space admits transitions between orthants
both through NNI moves and through \emph{rank changes}—interchanges in the chronology of two speciation events.  Extensions of this framework to phylogenetic
networks~\cite{HuberMoultonMurakami2023} and to discrete time-trees~\cite{GavryushkinWhiddenMatsen2018}
have since been introduced.

In this paper we introduce and study \emph{$\sigma$-space}, a moduli space for pairs of
ultrametric phylogenetic trees $(T_H, T_P)$ in which $T_P$ can be mapped into $T_H$ via the \emph{parasite map} $\ell: P \to H$ that sends
parasite species to their appropriate host species, and extends uniquely to a time-respecting immersion of the parasite tree inside the host tree. 
Such nested pairs encode reconciled
host--parasite (or species--gene) evolutionary histories without host-switching.
We build $\sigma$-space as a cube complex from orthants corresponding to resolved nested ranked tree topologies, and parametrise them using the refinement of $\tau$-coordinates: $\{t_i\} = \{t^H_j\}\cup\{t^P_k\}$ and $\sigma_i = t^{s_i} - t^{s_{i-1}}$.
We moreover give a detailed description of low-dimensional $\sigma$-spaces.
A key input in our approach is the \emph{nesting sequence} 
$S \in \{H,P\}^{n+m-2}$
recording the subsequence of speciation
events corresponding to each tree. 

Our main result is that, given any number of parasites and hosts and mapping of the parasite into the host leaves, the
link of the origin in the corresponding $\sigma$-space contains no
     3-cycles (Proposition~\ref{prop:no3-cycle}), and we use Gromov's cube condition
     to conclude that:
\begin{thmx}\label{introthm:CAT(0)}[\cref{{thm:CAT(0)}}]
$\sigma$-space is CAT(0).
\end{thmx}     
     We therefore can compute unique geodesics and Fr\'echet means in $\sigma$-space.

We furthermore characterise which nesting sequences are \emph{admissible} for a given
    map on leaves (Lemma~\ref{lem:nestingseq}), and identify a poset
    structure on admissible nesting sequences interpolating between the
    \emph{interleaved} case (maximally cospeciating) and the \emph{decoupled}
    extreme (only duplication events).

We study the geometry of $\sigma$-space showing that cospeciation events correspond to external boundaries, and identify the \emph{domain of perfect cospeciation} as a geometrically distinguished locus corresponding to the maximal intersection of cospeciation boundaries.

The closest precursor to our work is the work by Huggins, Owen, and Yoshida~\cite{huggins2012}, who study a space of cophylogenetic trees as a subspace of a product of BHV-spaces. We relate our work to theirs (and \cite{ultra}) via the forgetful map $F$, which forgets the parasite map and sends a nested tree to the product of individual $\tau$-spaces for the host and the parasite tree. This is a homeomorphism onto a subspace of pairs of trees that can be nested,
with partial inverse given by appending $\ell.$
Our work goes beyond theirs in giving a combinatorial and geometric description of the co-phylogeny subspace, defining an intrinsic metric, and adopting the ultrametric assumption, which may be more appropriate to phylogenies with different evolutionary scales.

\subsection*{Acknowledgements}
GG was supported by NSF MSPRF 2202895.
RH was supported by the Dutch Research Council (NWO) through the grant VI.Veni.212.170.

\section{Phylogenetic trees and their moduli spaces}\label{sec:tree}

We briefly recall some definitions of phylogenetic trees and the geometric models used to compare them. 

\subsection{Phylogenetic trees}
 
\begin{definition}
  An (ultrametric) \emph{phylogenetic tree} $T$ is a finite weighted tree with a distinguished root,  no other degree 2 vertices, and all degree--$1$ vertices (the \emph{leaves}) are
  \begin{itemize}
      \item bijectively labelled by a set $L$ of distinct taxa.
      \item A fixed path distance $h$ from the root.
  \end{itemize}
  $h$ is called the {\em height} of the tree.
\end{definition}
Ultrametric trees are completely determined by the the {\em path distance} on the leaves; that is, $d_T(l_1,l_2)$ is the sum of the edge weights on the (unique) path from $l_1\to l_2$ in $T$, so that $d_T(i,j)$ is twice the height of the most recent common ancestor (MRCA) of $i$ and $j$. By a classical result, the set of $d_T(i,j)$ values is a finite metric which satisfies the {\em ultrametric condition}:

\begin{definition}
    A metric $d_T:L\times L \rightarrow \mathbb{R}_{\geq 0}$ on a finite set $L$ is {\em ultrametric} if it satisfies the strong triangle inequality
\[
d_T(x,y) \le \max\{d_T(x,z), d_T(y,z)\} \qquad \text{for all } x,y,z \in L.
\]

\end{definition}
The tree $T$ can conversely be seen as the geometric or graphical {\em realisation} of $d_T$. Abusing notation slightly, we will use $d_T$ to represent both the metric and its matrix of pairwise distances.
For ultrametric trees we can think of the leaves as having a \emph{time} $t_0,$ which is often 0 by convention to represent the present, the root as having time $h$ (i.e. $h$ years ago), and each point $x\in T$ (including the edges and their interior) has time \[t_x = t_0 + \min_{l\in L_T} \{d_T(l,x)\}\leq h.\]
This gives the tree $T$ and all of its edges an orientation toward the root and backward in time. The \textit{leaf path} is the unique path $\gamma_l:[0,\infty]\to T$ from each leaf toward the root going linearly backward in time at unit speed (increasing in height) (\cref{fig:tree_ex}). 
Each node has precisely one edge oriented outward toward the root, and we can also consider an implied infinite edge from the root extending up backwards in time. 

We can describe the {\em descendants} of a node or edge as all leaf labels that are connected to it by an oriented path. The {\em in-degree} of a node $v$ is the number of incident edges oriented toward $v$; it also represents the number of populations speciating at $v.$
\begin{definition}
A \emph{ranked tree topology} on $n$ taxa $rt_n$ consists of an unweighted rooted tree on leaves $[n]:=\{1,2,\dots n\}$, together with a total ordering, the {\em rank}, on internal nodes that respects the time orientation of all edges toward the root, which always has the highest rank. It can also be characterised by a sequence of partitions associated to each internal node:
\[\{x_1 | x_1' |\dots |x_1^{(d_1)}\}, \{x_2 | \dots| x_2^{(d_2)}\}\dots,\{x_N | \dots | x_N^{(d_N)}\},\]
where for each node $v$ with ranking $i$, $d_i$ is the in-degree $\deg(v)-1$ of $v$, and the partition, or {\em split} $\{A | A' |\dots |A^{(d_v)}\}$ represents the descendant sets along each in-edge of $v$.  We have in all cases that the final split comes at the root, with $\cup_i x^{(i)}_{N} = [n]$. 

If each internal node has in-degree 2, we say $rt_n$ is binary or {\em fully resolved}. 
\end{definition}

 Each ranked topology $rt_n$ is an equivalence class of trees under label- and rank-preserving homeomorphisms. Conversely, phylogenetic trees can be fully described by their ranked topology $rt_n$
together with the time steps \[t_1\leq t_2\leq \dots \leq t_{n-1}\leq h\] of each node.

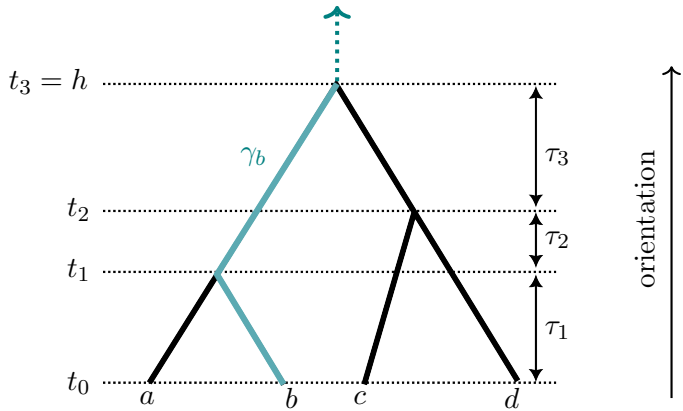
\begin{figure}
    \centering
	    \begin{tikzpicture}[scale=1]
        \node at (-3,0)  {\includegraphics[width=0.4\linewidth]{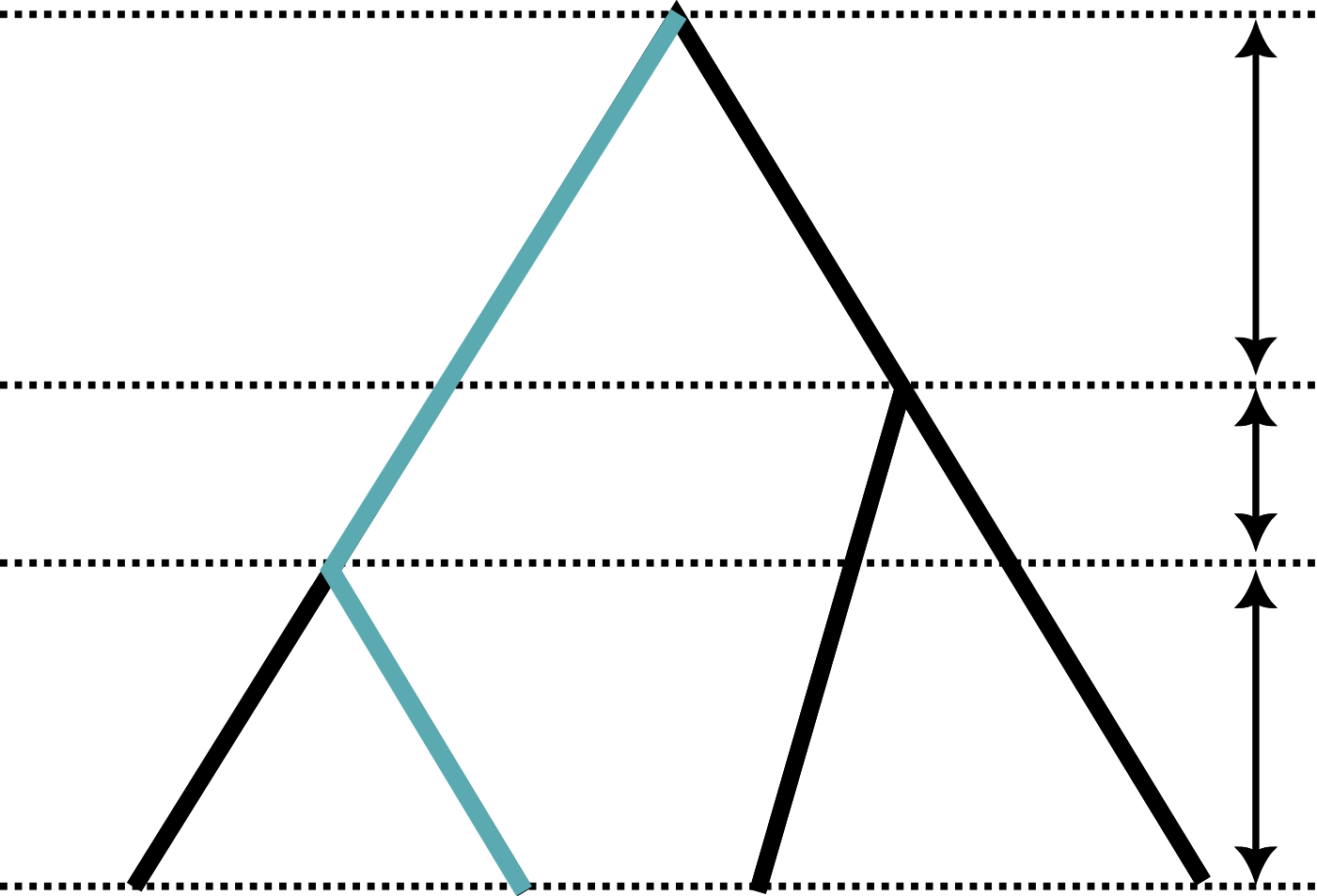}};
       \node at (-5.4,-2.2) {$a$};
        \node at (-3.5,-2.2) {$b$};
        \node at (-2.6,-2.2) {$c$};
        \node at (-0.6,-2.2) {$d$}
        ;
        \node at (-6.3,-2) {$t_0$};
       \node at (-6.3,-0.5) {$t_1$};
        \node at (-6.3,0.3) {$t_2$};
        \node at (-6.7,2) {$t_3=h$}
        ;
         \node at (0,-1.3) {$\tau_1$};
         \node at (0,-0.1) {$\tau_2$};
          \node at (0,1) {$\tau_3$}
          ;
        \node at (-4,1) {\textcolor{teal}{$\gamma_b$}};
    \draw[->, thick] (1.5,-2.2) -- (1.5,2.2)
    node[pos=0.3, right, xshift=-10pt, rotate=90] {orientation};
    \draw[->, teal, dotted, line width=1.5pt] (-2.9,2) -- (-2.9,3);
\end{tikzpicture}

    \caption{Example rooted ultrametric tree illustrating speciation times, $\tau$-coordinates and the leaf path $\gamma_b$ for the leaf labeled $b$.
    }
    \label{fig:tree_ex}
\end{figure}

\subsection{$\tau$-space of ultrametric trees}
In the current paper, we will be inspired for our construction of a space of reconciled cophylogenetic trees by the $\tau$-space parametrisation of tree space from \cite{ultra}.

Let $RT_n$ be the set of all ranked topologies of fully resolved phylogenetic trees on $n$ taxa, where fully resolved now means that the tree is binary and its internal nodes all have different height. There are \[\prod_{k=2}^{n} \binom{k}{2}=\frac{(n-1)!n!}{2^{n-1}}\] such ranked topologies, 
as seen easily by considering pairing (coalescing) lineages upwards in time. 
The $\tau$-space parametrisation of \cite{ultra} describes phylogenetic trees using their ranked topology and the differences \[\tau_i = t_i-t_{i-1}\]  of node times for $i = 1, \dots, n-1$ (see \cref{fig:tree_ex}). A tree with $n$ leaves is described by a pair $(rt(T), \overline{\tau})$ with $rt(T) \in RT_n$ and $\overline{\tau}\in \mathbb{R}_{\geq 0}^{n-1}$. 
The space of all such trees $\T_n$ is then constructed from the positive orthants $\mathbb{R}_{\geq 0}^{n-1}$ for every ranked topology by gluing along the faces where some of the coordinates $\tau_i$ are zero, in such a way that the ranked tree topologies of the orthants agree. 
The origin of all orthants are identified as it corresponds to the degenerate tree with height 0.

If we bound all $\tau_i$'s by a large positive constant, we can now consider $\T_n$ as a cubical complex. Such bounds in this and later spaces will be irrelevant for our results in the paper, and we will interchangeably use the terms cube and orthant.

$\T_n$ is naturally endowed with a metric: a distance $g(T_1,T_2)$ between two phylogenetic trees. When $T_1$ and $T_2$ have the same topology, this corresponds to taking the norm of the difference of their parametrisation vectors $\overline{\tau}$. When the topologies differ, we rely on the connectivity of the space. We define the metric $g$ of $\T_n$ as the path length infimum in $\tau$-space.
The pair $(\mathcal{T}_n, g)$ is then a metric space, and we can ask about its metric properties. 
It was shown in \cite{ultra}, in analogy with the same result for BHV space, that $\tau$-space is $CAT(0)$, or non-positively curved, as a cube complex. This implies that a unique geodesic exists between any two trees, which allows for the definition of Fr\'echet means among other statistical results. Geodesics in $\tau$-space can be computed in polynomial time using a modification of the algorithm for $\BHV$ space in \cite{owen2010fast}.

Since $\tau$-space, like $\BHV$ space, is formed by gluing contractible spaces along contractible subspaces, the spaces are contractible. However, their non-trivial topology becomes visible by looking at the link of the origin.
\begin{definition}
    The link of a vertex in a cube complex is the simplicial complex obtained by intersecting the complex with a small sphere around the vertex.
  \end{definition}

Transitions from one orthant to another through a codimension 1 face are described by a \textit{nearest neighbor interchange (NNI)} in the tree topology, illustrated in Figure \ref{fig:NNI}, or if the corresponding ranked topologies are related by a \textit{rank change} (Figure \ref{fig:rankint}), where the tree topology is unchanged but two nodes exchange their order. Codimension 1 faces in $\T_n$ can be part of one, two, or three cubes. Setting $\tau_1= 0$ corresponds to an external boundary of the cube complex. 
If two internal nodes have the same rank, then this can be resolved in two ways, thus a rank change corresponds to the adjacency of two cubes. A degree 4 (in-degree 3) node in the tree can be resolved into a binary tree in three ways corresponding to choosing the more closely related pair of the three downwards branches, thus an NNI corresponds to three cubes meeting at an $(n-2)$-dimensional face. In general, a codimension $k$ face of an orthant corresponds to the vanishing of $k$ of the $\tau$-coordinates.

The $\tau_1$ coordinate does not interact with the topology of the tree since it only captures the time between the last node and the leaves. Omitting $\tau_1$ gives rise to the $(n-2)$-dimensional space $\T_n^0$. Adding $\tau_1$ back in simply corresponds to a product: $\T_n = \T_n^0\times \mathbb{R}_{\geq0}$.

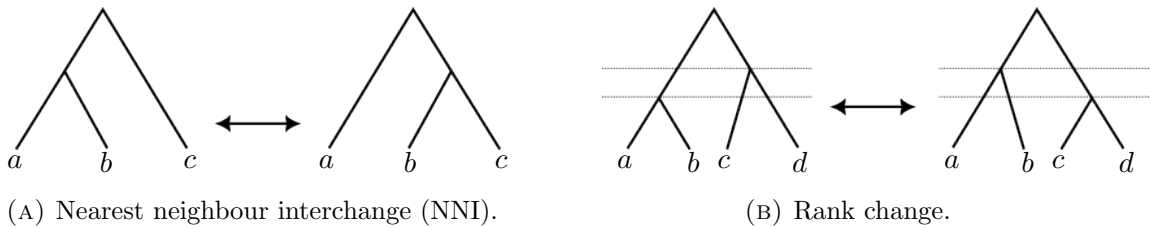
\begin{figure}[h]
    \centering 
    \begin{subfigure}[t]{0.45\textwidth}
        \centering
	    \begin{tikzpicture}[scale=1]
        \node at (-3,0)  {\includegraphics[height=0.28\textwidth]{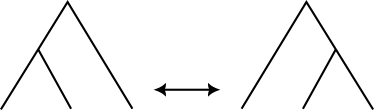}};
        \node at (-6.2,-1.1) {$a$};
        \node at (-5,-1.1) {$b$};
        \node at (-3.9,-1.1) {$c$};
        \node at (-2.1,-1.1) {$a$};
        \node at (-1,-1.1) {$b$};
        \node at (0.2,-1.1) {$c$};
\end{tikzpicture}
        \caption{Nearest neighbour interchange (NNI).}
        \label{fig:NNI}
    \end{subfigure}%
    \hspace{1cm}
    \begin{subfigure}[t]{0.45\textwidth}
        \centering
        	    \begin{tikzpicture}[scale=1]
        \node at (-2.5,0)  {\includegraphics[height=0.28\textwidth]{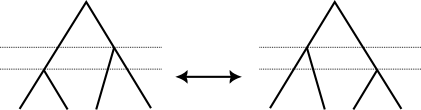}};
        \node at (-5.8,-1.1) {$a$};
        \node at (-4.9,-1.1) {$b$};
        \node at (-4.5,-1.1) {$c$};
        \node at (-3.5,-1.1) {$d$};
        \node at (-1.5,-1.1) {$a$};
        \node at (-0.5,-1.1) {$b$};
        \node at (-0.1,-1.1) {$c$};
        \node at (0.8,-1.1) {$d$};
\end{tikzpicture}
        \caption{Rank change.}
        \label{fig:rankint}
    \end{subfigure}
    \caption{Moves on ranked tree topologies.}
    \label{fig:orthantmoves}
\end{figure}

\section{Nested trees}\label{sec:nestedtrees}

We now introduce analogues of the definitions in \Cref{sec:tree} for nested pairs of phylogenetic trees, that is, reconciled trees of hosts and associates/parasites, in order to define a moduli space of nested trees $\Sigma$ in \Cref{sec:nestedspace}.

\subsection{Nested ultrametric trees}
\begin{definition}
    Consider two ultrametric distance matrices $d^H$ and $d^P$, and a set map on leaves $\ell : L(P)\rightarrow L(H)$. We say that the realizations $T_H$ and $T_P$ are {\em host-parasite compatible along $\ell$} if 
\begin{eqnarray}
    \label{eqn:metricmatrix}  
    &&d^P(i,j) \geq d^H(\ell(i),\ell(j))\hspace{4em}\mbox{ for all }\hspace{0.5em}i, j\in L(P).
\end{eqnarray}

\end{definition}
Note that this definition is not symmetric, and it implies that the height $h_P\geq h_H.$
\begin{definition}\label{def:nestedtree}
    A (once) \textit{nested phylogenetic tree} is a tuple $(T_H,T_P,\ell)$ such that $T_H$ and $T_P$ are host-parasite compatible along $\ell$. We call $T_H$ the {\em host tree}, $T_P$ the {\em parasite tree}, and $\ell$ the {\em parasite map} that maps leaves of the parasite tree to their host species in $L(H)$. The {\em height} of $(T_H,T_P,\ell)$ is the height of $T_P$.

    We refer to such a tree as being of {\em type} $(n,m,\ell)$.

\end{definition}

   If $T_H,T_P$ are host-parasite compatible along $\ell,$ then $\ell$ extends uniquely to a map \[\hat{\ell}:T_P\rightarrow T_H.\] 

   \begin{definition}   
   For $x\in T_P,$ let $l\in L(T_P)$ be any descendant of $x$ (without loss of generality due to ultrametric condition) and take the unit speed path from the leaf through the root $\gamma_l:[0,\infty]\to T_P$ such that $\gamma_l(0) = l$ and $\gamma_l(h_x) = x$, where $h_x$ is the height of $x$. Then define $\hat{\ell}(x):=\gamma_{\ell(l)}(h_x)\subset T_H$. We call $\hat{\ell}$ the \textit{realisation} of the parasite map.
   \end{definition}  

Let $T$ be a nested phylogenetic tree that is \textit{fully resolved}, by which we mean that both $T_H$ and $T_P$ are each fully resolved and none of their nodes have the same time. For $n$ and $m$ the respective number of non-root leaves, we denote by 
\[t^H_1 < t^H_2 < \dots <t^H_{n-1} <  h\]
the ordered heights of internal nodes in the host tree, and by
\[t^P_1 < t^P_2 < \dots <t^P_{m-1} = h\]
the ordered heights of internal nodes in the parasite tree. Let $N = n + m$
Together, the nested tree has an ordering
\[t^{s_1} < t^{s_2} < \dots < t^{s_{N-2}}=h\]
where each $t^{s_j} \in \{t^P_i\}\cup \{t^H_i\}$.
\subsection{Nesting sequences}
\begin{definition}
    For a fully resolved nested tree with $t-$coordinates $\{t^H_i\}$ and $\{t^P_j\}$, the
    \emph{nesting sequence} $S \in \{P,H\}^{N-2}$ is the sequence such that $S_i = P$ if $t^{s_i}$ is the time of a speciation event in $T_P,$ and $S_i = H$ if $t^{s_i}$ is the time of a speciation event in $T_H.$
\end{definition}

If the parasite map is injective when restricted to leaves, then the nesting sequence of any resolved nested tree starts with $S_1 = H$ and ends with $S_{N-2} = P$. 
Moreover, in this case at any point in the nesting sequence, there are equally many or more H's than P's to the left. We now generalise this statement for the non-injective case.
\begin{definition}
We call $\dhh=|\ell(L_P)|\leq n$ the \emph{host degree} and $\dpp = m- \dhh$ the \textit{parasite multiplicity}.
\end{definition}
The host degree is the number of host leaves that contain at least one parasite and the parasite multiplicity is the number of merges that can occur in the parasite tree before the first coalescence in the host tree occurs.

Given a tuple $(n,m,\ell)$, we call a nesting sequence $S$  {\em admissible} if it occurs as the nesting sequence of any nested tree of type $(n,m,\ell)$.  
\begin{lemma}\label{lem:nestingseq}
A nesting sequence $S$ is admissible for trees of type $(n,m,\ell)$ if and only if for any $i\in \{1, \dots , n+m-2\}$, 
\begin{equation}
    |\{S_j =P\}_{j\leq i}|\leq |\{S_j = H\}_{j\leq i}|+\dpp.\label{eq:nestingcond}
\end{equation}
\end{lemma}
\begin{proof}
Suppose that $S=P^{\dpp}(HP)^{\dhh-1}H^{n-\dhh}$, which is the minimal sequence satisfying \cref{eq:nestingcond}. Then there is a nested tree with this nesting sequence, in which all parasites coalesce first until there are isolated parasite lineages in $\dhh$ host lineages. A host merge among these host lineages enables a further parasite merge et cetera, until there is a single remaining parasite lineage. Above this the empty host lineages merge.
Any scenario with more P's to the right is a possible nesting sequence for a nested tree with the same host and parasite tree topologies by changing the rank of parasite and host nodes. Since the first $P^{\dpp}$ exhaust additional parasite lineages, there is no nested tree with more $P's$ appearing to the left in the nesting sequence.
    \end{proof}

\begin{definition}
    The nesting sequence $S_{int}=P^{\dpp}(HP)^{\dhh-1}H^{n-\dhh}$ is called \textit{interleaved} and the nesting sequence $S_{dec}=H^{n-1}P^{m-1}$ is called \textit{decoupled}.
\end{definition}
See \Cref{fig:treesinS3} for an example. Note that if $\ell$ is bijective then the interleaved nesting sequence is $(HP)^{n-1}$.

\begin{definition}
    We define a poset $(\mathcal{NS}, \leq)$ of admissible nesting sequences for a fixed type $(n,m,\ell)$ as follows.
    We say $S\leq  S'$  
if there is some $i$ such that $S_i = P, S_{i+1}= H, $ and $S'$ is identical with these entries reversed: 
\[S' = \left\lbrace \begin{array}{ll}
    S_j & j\notin \{i,i+1\} \\
     H & j = i\\
     P & j=i+1.
\end{array}\right.
\]
We extend to the poset by taking the reflexive transitive closure.
\end{definition}
Note that passing from $S$ to $S'$ in the definition corresponds to a rank interchange passing a parasite speciation above the next host speciation. By Lemma \ref{lem:nestingseq}, if $S$ is admissible, then $S'$ is.
\begin{lemma}\label{lem:upset}
Let $nrt(T)$ be a nested ranked topology with individual ranked topologies $rt_H$ and $rt_P$, leaf map $\ell$
and nesting sequence $S(T).$ Let $S'$ be any $S'\geq S(T) $ in the nesting sequence poset, then there is a compatible nested ranked topology $(rt_H, rt_P, \ell, S')$.
\end{lemma}
\begin{proof}
Follows from \Cref{lem:nestingseq}:
\[|\{S'_j=P\}_{j\leq i}|\leq |\{S_j=P\}_{j\leq i}|\leq
|\{S_j=H\}_{j\leq i}|+\dpp \leq |\{S'_j=H\}_{j\leq i}|+\dpp.\]
\end{proof}
Hence trees can freely move in the direction of decoupling (``up-sets" of $S$) via the rank interchanges encoded by the relation $\leq$ of the nesting poset.  
The interleaved and decoupled nesting sequences are the unique minimal resp. maximal element of $(\mathcal{NS}, \leq)$. 

If the nested tree is not fully resolved, then the inequalities in the sequence of $t^{s_i}$'s are no longer strict, and the nesting sequence as above is not unambiguously defined.

There is a refinement of the nesting sequence of a fully resolved nested tree that captures additional biologically relevant information, namely whether the parasite events can be thought of as a cospeciation or a duplication.
\begin{definition}
    Let $\{x_p|x_p'\},\{x_h|x_h'\}$ be splits corresponding to speciation events in the parasite and host tree, respectively. Then we say that the events are {\em coupled} if $\ell(x_p)\subseteq x_h,\ell(x_p')\subseteq x_h'$ (or $\ell(x_p)\subseteq x_h'$ and $\ell(x_p')\subseteq x_h$). 
\end{definition}
\begin{definition}
    For a fully resolved nested tree with $t-$coordinates $\{t^H_i\}$ and $\{t^P_j\}$, the
    \emph{annotated nesting sequence} $S \in \{P^c, P^d,H\}^{N-2}$ is the sequence such that
    \begin{itemize}
        \item  $S_i = H$ if $t^{s_i}$ is the time of a speciation event in $T_H$.
        \item $S_i = P^c$ if $t^{s_i}$ is the time of a speciation event in $T_P$ that is coupled with the host speciation at $S_{i-1}=H$.
         In this case $t^{s_i}\rightarrow t^{s_{i-1}}$ corresponds to a \textit{cospeciation event} of parasite and host.
        \item $S_i = P^d$ if $S_{i-1}\neq H$ or if the lineages spawning at $t^{s_i}$ are not split at $t^{s_{i-1}}.$ In this case $t^{s_i}$ corresponds to a \textit{duplication event} of the parasite inside the host.
    \end{itemize}   
\end{definition}
Note that in this construction we have defined {\em cospeciation} as both a coupling of speciation events and adjacency in the combined ranking. 

\begin{definition}
    Let $(T_H,T_P,\ell)$ be a nested tree that is not necessarily fully resolved.
     Let $\{t_k\}_{k=1,\dots,K}$ be the union of timesteps in $T_H$ and $T_P$.
 At time $t_k$, suppose $r$ host and $s$ parasite lineages speciate, with in-degree $d^H_1,\dots, d^H_r$ and $d^P_1,\dots, d^P_s$, resp. Let $r_{tot}= \sum_{i=1}^r d^H_i$. Define $C$ to be the set of parasite speciation events at $t_k$ that are coupled to some host speciation also occurring at $t_k.$ Let $$s_{c} := \sum_{i\in C} d^P_i \hspace{1cm} s_d := \sum_{i\notin C} d^P_i.$$
  We construct a substring $\mathcal{S}_k$ of the nesting sequence: 
    \[\mathcal{S}_k:=H^{r_{tot}}(P^c)^{s_c}(P^d)^{s_d}.\]    
 Finally, we define the canonical admissible nesting sequence of $(T_H,T_P,\ell)$ to be the concatenation \[S = \mathcal{S}_1\dots\mathcal{S}_k\dots \mathcal{S}_K\]
    
\end{definition}
From this, we can define nested ranked topologies $(rt_H,rt_P,\ell,S)$ for ranked topologies and orderings that are not fully resolved.

\subsection{Nested tree topology}

\begin{definition}
        A \textit{nested ranked tree topology} is a tuple \[(rt_H, rt_P, \ell, S)\] of a host ranked tree topology $rt_H$, a parasite ranked tree topology $rt_P$, a leaf map $\ell: L_P \rightarrow L_H$ and a compatible nesting sequence $S$.
        \end{definition}

A nested tree homeomorphism is defined as a map $\phi:(T_H,T_P,\ell)\to (T_{H'},T_{P'},\ell')$ such that $\phi|_{T_H}, \phi|_{T_P}$ are label-preserving homeomorphisms onto their image, $\phi$ preserves the total rank of vertices in $T_H\cup T_P$, and $\phi$ commutes with parasite maps: \[\phi|_{T_H}\circ\hat{\ell}(T_P)=\hat{\ell'}\circ \phi_{T_P}(T_P).\] Then we have that a nested ranked tree topology $nrt$ is an equivalence class of nested phylogenetic trees under nested tree homeomorphism.
          
Not all tuples $(rt_H, rt_P, \ell, S)$ are admissible. However, for every tuple $(rt_H, rt_P, \ell)$, there exists at least one compatible nesting sequence, namely the decoupled nesting sequence. Depending on the compatibility of $(rt_H, rt_P, \ell)$, more interleaved nesting sequences (that are lower down in the poset $(\mathcal{NS}, \leq)$) may be admissible.

\section{$\sigma$-space of nested trees}\label{sec:nestedspace}
We will construct a metric space $\Sigma_{n,m,\ell}$ of nested phylogenetic trees using the descriptions of nested trees from \Cref{sec:nestedtrees} to assemble a cubical complex which inherits a global metric from its piecewise construction.

\subsection{The construction of $\sigma$-space}

In analogy with $\tau$-space of \cite{ultra}, we define a parametrisation of nested trees by considering their ranked topology and the intervals between speciation times. 

Like individual phylogenetic trees, a nested tree is fully determined by its $nrt$ and an increasing time sequence $\{t^{s_i}\}_{s_i\in S}$. For $T$ a nested tree, not necessarily resolved, we define its $\sigma$-coordinates
\[\sigma^{S_i}_i = t^{s_{i}}- t^{s_{i-1}}\] where we set $t^0=0$ (see \cref{fig:orthantmoves}). The label $S_i\in \{P,H\}$ from the nesting sequence represents whether the node at the top of the interval $\sigma_i$ is a speciation event in $T_P$ or $T_H$. This label is unambiguous precisely when $\sigma_i$ is nonzero.
\begin{figure}[h]
    \centering 
    \begin{subfigure}[t]{0.45\textwidth}
        \centering
	    \begin{tikzpicture}[scale=1]
        \node at (-3,0)  {\includegraphics[width=0.95\textwidth]{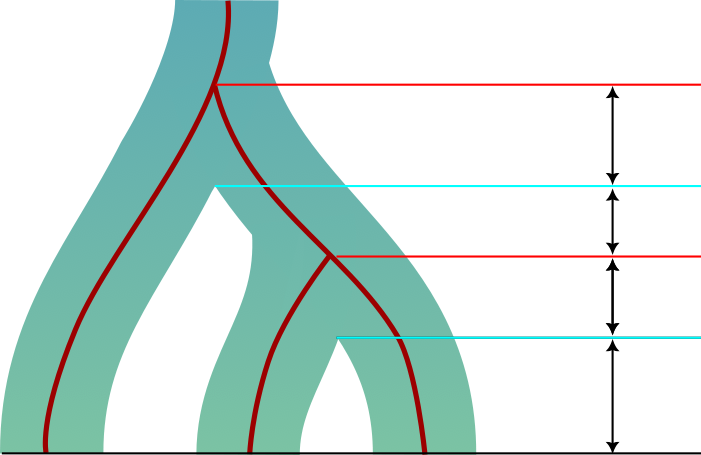}};
        \node at (-0.2,-1.5) {$\sigma_1^H$};
        \node at (-0.2,-0.62) {$\sigma_2^{P^c}$};
        \node at (-0.2,0.1) {$\sigma_3^H$};
        \node at (-0.2,0.8) {$\sigma_4^{P^c}$};
\end{tikzpicture}
        \caption{Interleaved tree in $\Sigma_3$ with nesting sequence $HPHP$.}
        \label{fig:HPHP}
    \end{subfigure}%
    \hspace{1cm}
    \begin{subfigure}[t]{0.45\textwidth}
        \centering
        	    \begin{tikzpicture}[scale=1]
        \node at (-2.5,0)  {\includegraphics[width=0.95\textwidth]{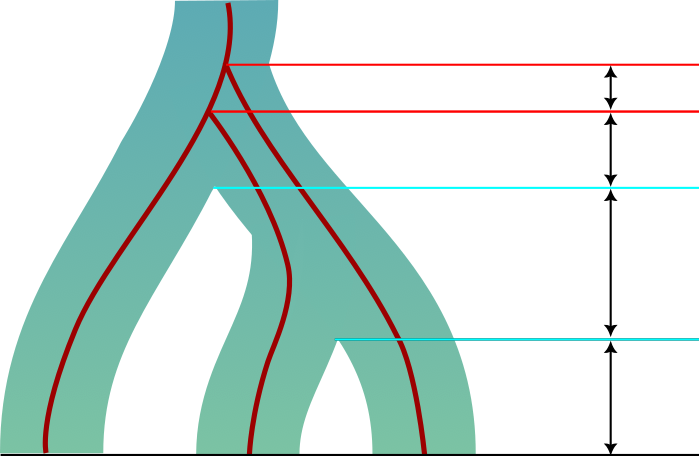}};
        \node at (0.3,-1.5) {$\sigma_1^H$};
        \node at (0.3,-0.3) {$\sigma_2^H$};
        \node at (0.3,0.7) {$\sigma_3^{P^c}$};
        \node at (0.3,1.3) {$\sigma_4^{P^d}$};
\end{tikzpicture}
        \caption{Decoupled tree in $\Sigma_3$ with nesting sequence $HHPP$.}
        \label{fig:HHPP}
    \end{subfigure}
    \caption{Interleaved (A) and decoupled (B) nested trees in $\Sigma_3$, together with their $\sigma$-coordinates.}
    \label{fig:treesinS3}
\end{figure}
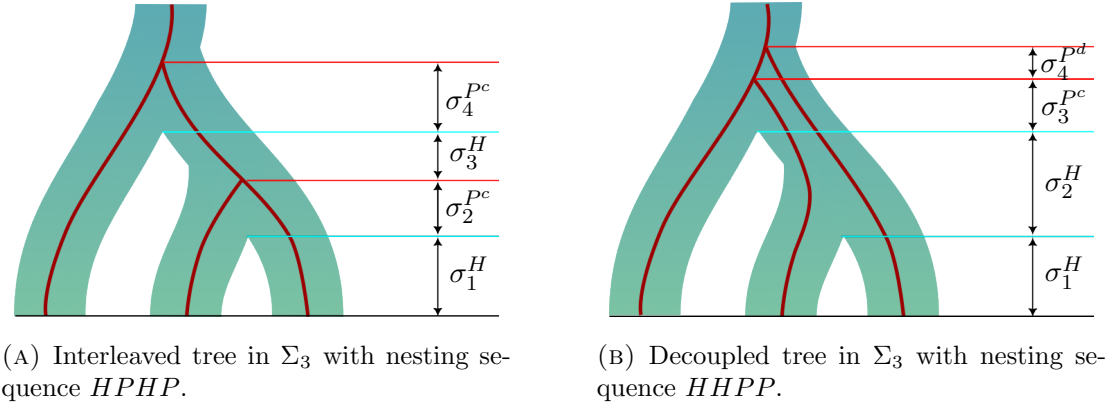

For a nested tree $T$ of type $(n,m,\ell)$, we parametrise $T$ by the pair $(nrt(T), \overline{\sigma})$ where $nrt(T)$ is its nested ranked topology and $\overline{\sigma}\in \Rp^{n+m-2}$ is its vector of $\sigma$-coordinates.
Say that there exist $M$ nested ranked topologies of type $(n,m,\ell)$\footnote{In future work we will approach the problem of counting the number of nested ranked topologies for a given nesting sequence. The problem of counting nested (ranked) tree topologies, or coalescent histories, has also been approached in \cite{Rosenberg2007, DegnanRosenbergStadler2012, DisantoRosenberg2015}.}.

 Let $N=n+m$. Then fully resolved nested trees correspond to the interiors of $M$ orthants, each isometric to $\Rp^{N-2}$. The set of orthants is partitioned by the possible nesting sequences of length $N-2.$
At a codimension $k$ face of the orthants, $k$ of the $\sigma$ coordinates go to 0. Each face corresponds to a nested tree topology which is not fully resolved, parametrised by the non-zero $\sigma$-coordinates.

The same nested ranked topology can arise as a boundary face of multiple different orthants, as trees which are fully resolved may degenerate to the same tree in multiple different ways. As a distinguished example, we highlight the {\em cone point}, the nested star tree with $\underline{\sigma}=\bf{0}$. The cone point is the unique dimension-0 face occuring in every orthant at the origin ${\bf 0}\in \mathbb{R}^{N-2}_{\geq 0}$. 

By identifying faces which represent the same trees (the indexing of $\underline{\sigma}$ may differ by a bijective correspondence but the $nrt$ is the same), we obtain a (connected) cube complex with the quotient topology, called $\Sigma_{n,m,\ell}$

We use this piecewise structure to endow $\Sigma_{n,m,\ell}$ with a metric. If two nested trees $T_0$, $T_1$ have the same nested tree topology then their distance is the difference of parametrisation vectors
\[g(T_0, T_1) = |\underline{\sigma}_0 - \underline{\sigma}_1|. \]
If they have different topology then $g(T_0, T_1)$ is the path length infimum.
\[g(T_0,T_1) = \inf_{\substack{\alpha:[0,1]\to \Sigma \\ \alpha(0) = T_0, \alpha(1) = T_1}}\limsup_{\substack{0=x_0\leq x_1\leq \dots \leq
x_k = 1\\ k\to \infty}} \sum_{i=0}^{k-1} g(\alpha(x_i),\alpha(x_{i+1}))\]

where $g(\alpha(x_i),\alpha(x_{i+1})) := 0$ if they are in two different cubical components.

\begin{lemma}
    $\Sigma_{n,m,\ell}$ is path-connected.
\end{lemma}
\begin{proof}
    The closure of every orthant is individually path-connected, and each contains  the cone point in its boundary.
\end{proof}
\begin{definition}
    \label{def:sigmaspace} $\Sigma_{n,m,\ell}$ is the length space consisting of the set of all nested trees of type $(n,m,\ell)$, i.e. with $n$ hosts, $m$ parasites, and a fixed map $\ell:[m]\to[n]$, together with the topology induced by the metric $g.$
\end{definition}

 If there is only one possible set map $\ell,$ for example if $n=1,$ then we write $\Sigma_{n,m}$. Otherwise, $\Sigma_{n,m}$ will denote an implied injective parasite map $\ell.$ Without loss of generality, we assume $\ell:[m]\hookrightarrow [n]$ is the identity map.

   If $m=n$ and $\ell$ is bijective then we write $\Sigma_n$.

As in the case of $\tau$-space, the coordinate $\sigma_1$ plays a less important role, and we can define a subspace $\Sigma_{n,m,\ell}^0$ of $\Sigma_{n,m,\ell}$ where trees are labelled by $rt(T)$ and the vector $(\sigma_2, \dots, \sigma_{n+m-2}) \in \Rp^{n+m-3}$. As before we have that $\Sigma_{n,m,\ell} = \Sigma_{n,m,\ell}^0 \times \Rp$.

At codimension 1 faces of orthants of $\Sigma_{n,m,\ell}$, analogous to in $\tau$-space, 1, 2 or 3 orthants can meet. 
$\tau$-space differs slightly in that $\sigma_1=0$ is not the only external codimension 1 face. 

Whenever a parasite speciation event at $t_i$ is followed directly by a coupled host speciation event at $t_{i+1}$ in the same lineage, then the face where $\sigma_{i+1} \rightarrow 0$ is an external face of the cube complex: 
\begin{definition}
A \emph{cospeciation boundary} is a codimension 1 facet of an orthant in $\Sigma_{n,m,\ell}$ corresponding to a cospeciation event, i.e. $\sigma_i^{P^c}=0.$
\end{definition}

\begin{definition}
The \emph{cospeciation boundary} of a ($k$-dimensional) orthant in $\Sigma_{n,m,\ell}$ is the subset of $\operatorname{cl}(\mathcal{O})$ corresponding to at least one simultaneous cospeciation event i.e. $\sigma_i^{P^c}=0$ for some $i.$
\end{definition}

\begin{lemma}
    A cospeciation boundary is external, in the sense that it is not shared by multiple orthants.
\end{lemma}
\begin{proof}
At the cospeciation boundary $\partial o$ of an orthant $o$, the nodes $i$ and $i-1$ have the same rank and correspond to a cospeciating parasite/host pair.
There is no nested tree topology that agrees with the tree topology of $o$ outside of the interval between nodes $i$ and $i-1$ while resolving $\sigma_i$ differently.
\end{proof}

A cospeciation boundary corresponds biologically to a cospeciation event.
As shown in Lemma \ref{lem:nestingseq}, the interleaved nesting sequence of $\Sigma_{n,m,\ell}$ is 
\[S_{int}=P^{\dpp}(HP)^{\dhh-1}H^{n-\dhh},\]
hence the maximum number of cospeciation events is $\dhh-1$. 

\begin{definition}
Let $\ell$ be surjective. For a nested tree $(T_H,T_P,\ell),$ consider the $n\times n$ pushforward matrix $\ell_*(d^P)_{kl} = \min_{\ell(i)=k, \ell(j) = l} d^P(i,j)$, and define two matrices $$M_{ij} = d^P(i,j) - \ell_*(d^P)_{\ell(i),\ell(j)}$$ $$ N = \ell_*(d^P)-d_H.$$ If $M = \mathbf{0}$, we say that $(T_H,T_P,\ell)$ is {\em concordant}; 
that is, the restriction of $T_P$ to $\ell^{-1}(h)$ for any host leaf $h$ is a star tree.
If $N = \mathbf{0},$ we say that $(T_H,T_P,\ell) $ is {\em maximally coupled}; meaning that all host speciation events happen simultaneously (cospeciate) with the earliest coupled parasite speciation (if there is one). \end{definition}

\begin{lemma}
Suppose $(T_H,T_P,\ell)$ with $\ell$ surjective and $T_H$ having distinct speciation times. Suppose it is maximally coupled, with nesting sequence $S$. Then $S'\lneq S $ is not admissible for $(rt(T_H),rt(T_P),\ell)$.   
\end{lemma}
\begin{proof}
Since $\ell$ is surjective, every host speciation has at least one coupled parasite speciation. As $(T_H,T_P,\ell)$ is maximally coupled, this means that every $H$ in the nesting sequence is followed by at least one cospeciating $P^c$, hence the nesting sequence after transposing $H$ and $P$, $\hat{S} = \dots S_i \dots (PH)\dots$, is no longer host-parasite compatible along $\ell$.
\end{proof}

In this case, $(T_H,T_P,\ell)$ is contained in the maximum number of cospeciation boundary components for its topology.

For $\ell$ bijective,
if $M=\bf{0}$ and $N=\bf{0},$ then $\hat{\ell}$ is an isomorphism. This corresponds with the following definition in the bijective case. 
\begin{definition}
    The \emph{domain of perfect cospeciation} is the union of codimension $\dhh-1$ faces in $\Sigma_{n,m,\ell}$ that are intersections of cospeciation boundaries.
\end{definition}

The domain of perfect cospeciation is a collection of faces containing the origin, so it is contractible. In the bijective case it moreover has a familiar topology.

\begin{lemma}
    If a nested tree topology $(rt_H, rt_P, \ell, S)$ is such that $\ell$ is bijective and $S$ is the interleaved nesting sequence $(HP)^{n-1}$, then the ranked tree topologies of host and parasite are \textnormal{matching}: $rt_H = rt_P$ in $RT_n$.
    \end{lemma}
\begin{proof}
    After any host coalescence, there is only one possible parasite merge, namely the one corresponding topologically to the host merge. Thus the two trees have the same ranked topology.
\end{proof}

\begin{lemma}
    If $\ell$ is bijective then the domain of perfect cospeciation embeds into $\Sigma_n$ as a copy of $\T_n$.
\end{lemma}
\begin{proof}
The only orthants intersecting the domain of perfect cospeciation have annotated nesting sequence $(HP^c)^{n-1}$ and matching host/parasite ranked tree topology. They moreover have corresponding edge lengths, hence the parasite and host tree are the same metric tree in $\T_n$.
\end{proof}
In particular, the intersection of the domain of cospeciation with the link of $\Sigma_n$ is homeomorphic to the link of $\T_n$.

\subsection{Low-dimensional $\sigma$-spaces}
We now give a full description of some $\sigma$-spaces with few hosts and parasites. 
As a convention we denote $\ell$ in terms of its partition of $L_P$, leaving out singletons.

\begin{example}[$\Sigma_{2}$]
    For illustration we start discussing the most basic example of two hosts an two parasites with a bijective parasite mapping. In this case, there is a single possible resolved nested ranked topology with nesting sequence $HP$. The space consists of a single orthant spanned by $(\sigma_1^H, \sigma_2^P) \in \Rp^2$. The face $\sigma_1 = 0$ is the usual external boundary while $\sigma_2 = 0$ is the cospeciation boundary of the orthant as well as the the domain of perfect cospeciation of the tree space.
\end{example}

\begin{example}[$\Sigma_{2,2,(12)}$]
    If we have two hosts and two parasites, but both parasites map to the same host, then we have two resolved nested ranked topologies separated by a rank change of the host and parasite speciation events. The two orthants have nesting sequence $HP$ (interleaved) and $PH$ (decoupled) and are glued together along the face $\sigma_2 = 0$ corresponding to a rank exchange. There is no cospeciation boundary.
\end{example}

\begin{example}[{$\Sigma_{2,3,(12)}$}]\label{ex:twoinone}
Suppose now that we have two host and three parasite leaves, with two parasites mapping to a single host - suppose that parasites 1 and 2 infest host $A$ while 3 is living on $B$. See \cref{fig:twoinone}. The possible nesting sequences are $PHP$ (interleaved) with one possible nested ranked topology and $HPP$ (decoupled) with three possible ranked topologies. We will discard $\sigma_1$ and describe the two-dimensional cube complex $\Sigma^0_{2,3,(12)}$.
The three decoupled $HPP$ orthants are glued along their $\sigma^P_3 = 0$ boundaries corresponding to the NNI in the parasite tree. Along the $\sigma^P_2 = 0$ boundary, the parasite tree with top split $1|23$ has a cospeciation boundary of split $2|3$ with host split $A|B$. The parasite tree with top split $2|13$ has a cospeciation boundary of split $1|3$ with host split $A|B$. However the parasite tree with top split $12|3$ has a rank change at $\sigma^P_2 = 0$ where it is glued to the interleaved $PHP$ orthant, where the parasite split happens inside of host $A$. For the interleaved orthant the $\sigma_3^P=0$ boundary is a cospeciation boundary of the top parasite split $12|3$ with $A|B$. The cone point $\sigma_2=\sigma_3=0$ corresponds to the cospeciation of the multifurcation of the parasite with the host split.
\end{example}
\begin{figure}[h!]
    \centering
	    \begin{tikzpicture}[scale=1]
        \node at (-3.2,0)  {\includegraphics[width=0.6\textwidth]{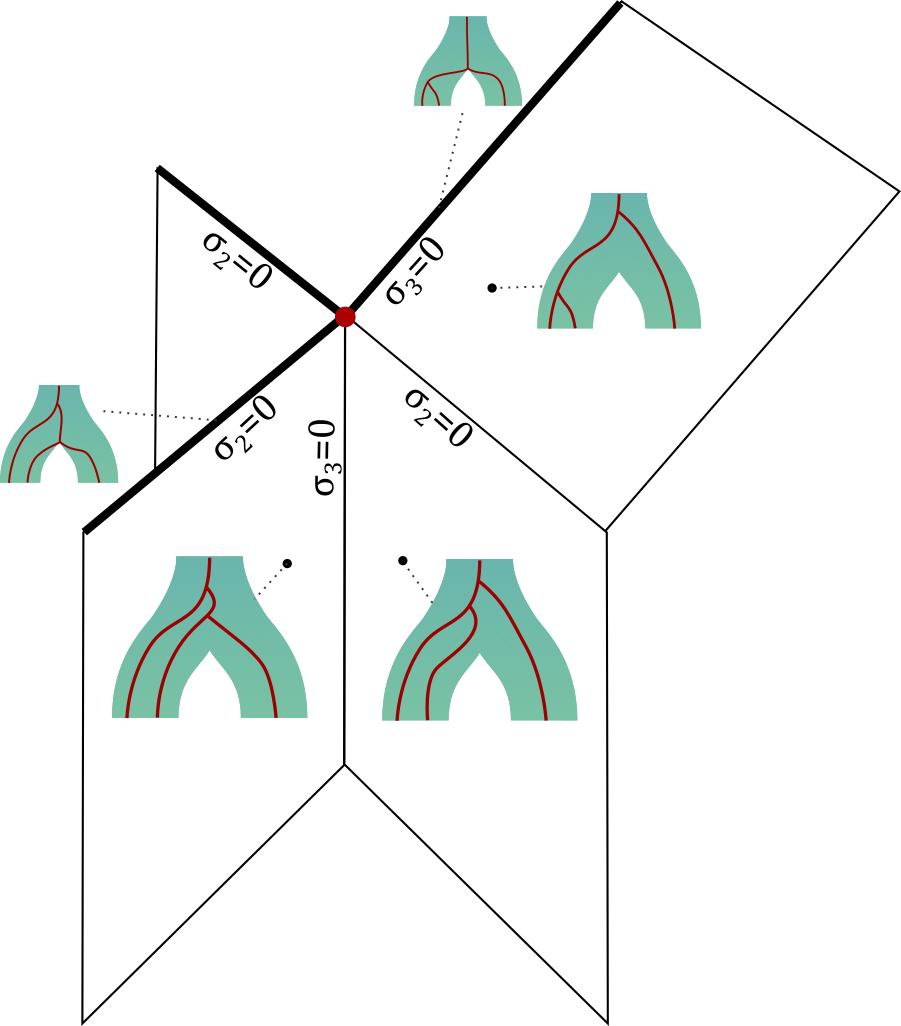}};
        \node at (-5.5,0) {$HPP$};
     \node at (-3,0) {$HPP$};
             \node at (-5.5,1.7) {$HPP$};
             \node at (-1.8,1.5) {$PHP$};
\end{tikzpicture}
  \caption{Illustration of the space $\Sigma^0_{2,3,(12)}$. Cospeciation boundaries are marked with a thick black line and the origin is marked with a red dot. Full description in \cref{ex:twoinone}.}
    \label{fig:twoinone}
\end{figure}
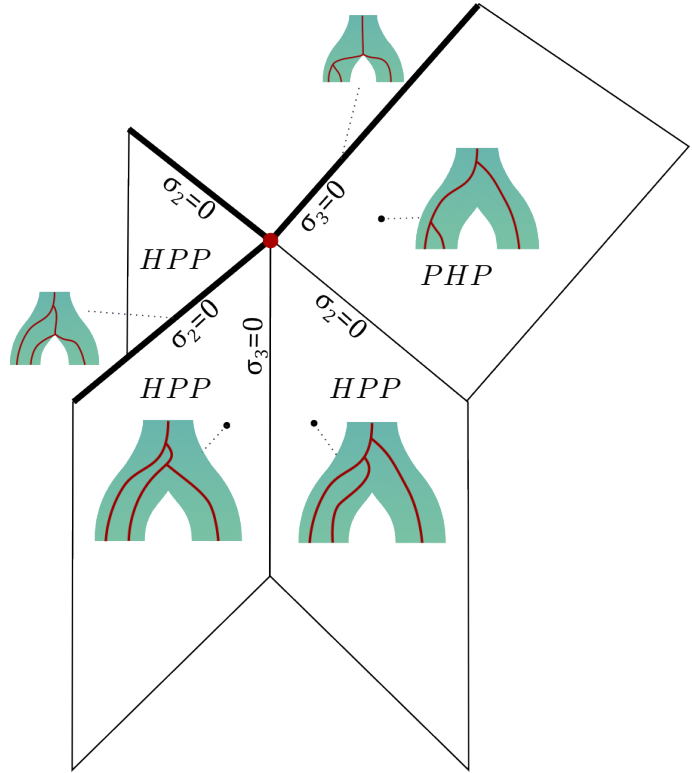

\begin{example}[$\Sigma_{3,2}$]
We now study the nested tree space for three hosts and two parasites, with an injective leaf map. Interestingly, this space turns out to be homeomorphic to the space $\Sigma_{2,3,(12)}$ studied before.
Suppose the leaf map sends 1 to $A$ and 2 to $B$. There are three decoupled $HHP$ tree topologies and one interleaved $HPH$ topology. The $\sigma_3^H=0$ boundaries of the decoupled orthants correspond to a cospeciation boundary of $1|2$ with $A|BC$, $AC|B$ and a rank change respectively, with the interleaved orthant attached along the rank change. The interleaved orthant has a cospeciation boundary of $1|2$ with $A|B$ at $\sigma_2^P=0$.
    \end{example}

\begin{example}[$\Sigma_{3}$]\label{ex:S3}

We now consider the case with 3 parasites, bijectively mapped to 3 hosts, see Fig. \ref{fig:treesinS3}. We restrict to describing the link of the origin in $\Sigma_{3}^0$ (Fig. \ref{fig:S3}).
The link of the origin is made up of one triangle in every three-dimensional orthant, with boundaries on the $\sigma_i=0$ planes for $i=2,3,4$. Both the parasite and host tree can have three topologies, denoted $A$, $B$ and $C$. We call the tree topologies of the host and parasite \emph{matching} if the leaf mapping gives a bijection on the splits of the two trees.
The link is glued from $3+3\cdot3 = 12$ triangles in the following way. Each of the three matched topologies from the interleaved ($HPHP$) orthants is connected by a rank exchange to a single decoupled ($HHPP$) orthant with matching topology. The matching decoupled topologies are then glued along their other two boundaries to 6 non-matching topologies related by an NNI in either the host or the parasite tree. The domain of perfect cospeciation is indicated by the three black dots, homeomorphic to the link of $\T_3$.

\end{example}

\begin{figure}[h!]
    \centering
           \includegraphics[width=\textwidth]{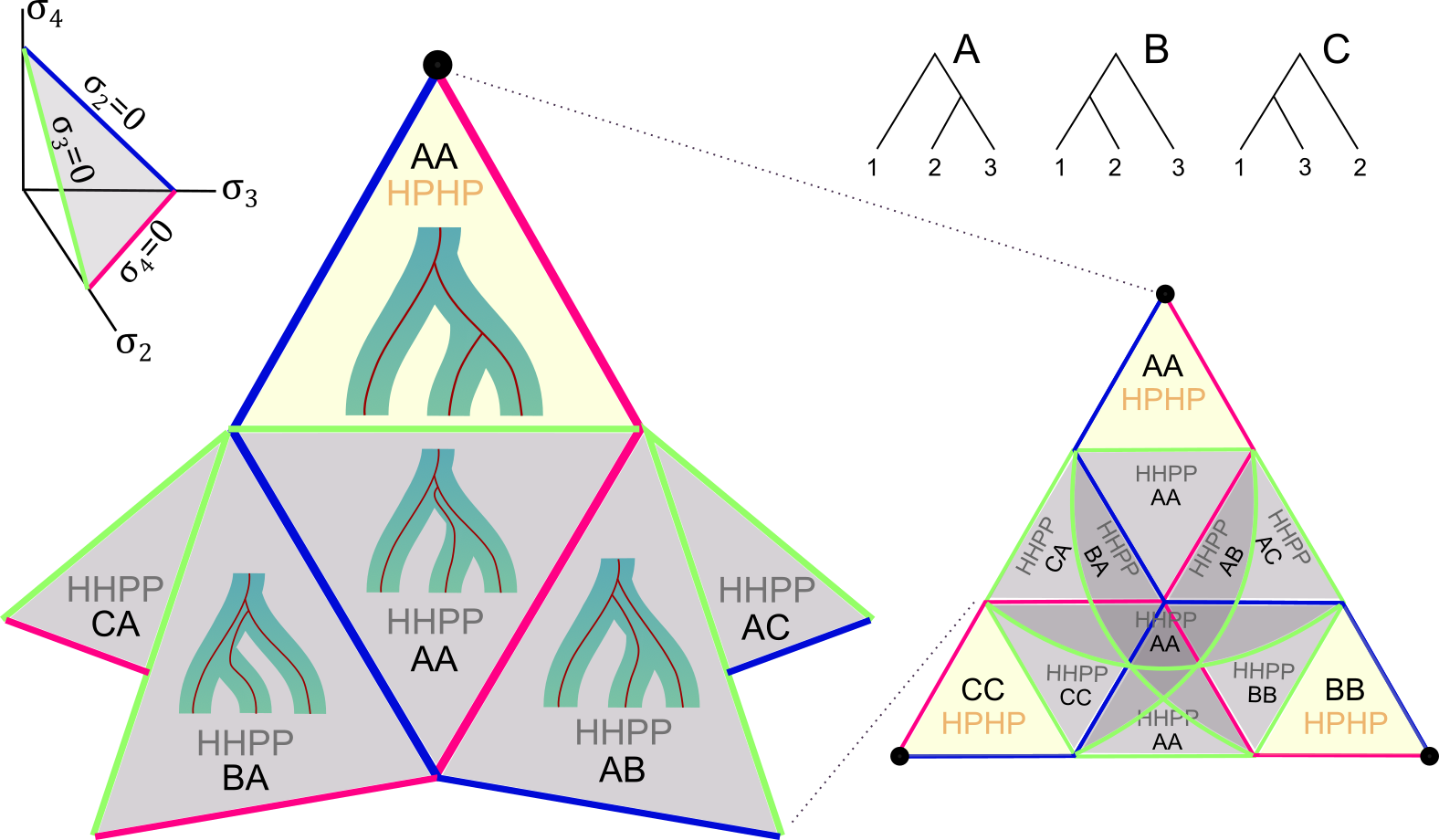}
        \caption{Illustration of the link of the origin in $\Sigma_3^0$. The link is composed of twelve triangles, of which three correspond to interleaved ($HPHP$) orthants with matching parasite and host tree topology and nine to decoupled ($HHPP$) orthants. All external faces are cospeciation boundaries. The domain of perfect cospeciation is indicated by three black dots.       
        Full description in \cref{ex:S3}.}
        \label{fig:S3}
\end{figure}

\subsection{Topological and geometric properties of $\Sigma$}
It is possible to define a space of ultrametric nested trees of type $(n,m,\ell)$ as a subspace of a product of $\tau$-spaces of ultrametric trees $\T_n \times \T_m$, where we restrict the fiber over $T\in \T_H$ to parasite trees that fit inside $T$ as a reconciled tree respecting the leaf map $\ell$.

Every nested phylogenetic tree admits two forgetful maps $F_H$ and $F_P$, which pointwise on the nested tree combine to a map $(F_H, F_P):(T_H, T_P, \ell) \rightarrow (T_H, T_P)$.  
These maps extend to continuous maps of tree spaces.

The $\tau$-space coordinates of $F_H(nrt,\underline{\sigma})$ and $F_P$ can be computed as follows. 
    Let $(T_H,T_P,\ell)$ be a generic nested tree, with coordinates $(nrt, \underline{\sigma})$ and nesting sequence $S_i.$ Let $i^P\subset [n+m-2]=\{1,2, \dots n+m-2\}$ be the subsequence of indices $i$ such that $S_{j} = P$ for $j \in i^P,$ and similarly define the complementary subsequence $i^H$ for $H$. We then define 
    \[\tau^P_j:= \sum_{i^P_{j-1}<k\leq i^P_j}\sigma_j,\]
    where $i^P_0 = 0$ by convention. Define $\underline{\tau}^H$ analogously.

Now we have the forgetful maps
\begin{align*}
    &F_P: \Sigma_{n,m,\ell} \rightarrow \T_m &F_H:\Sigma_{n,m,\ell}\rightarrow \T_n\\
    &(nrt, \underline{\sigma}) \mapsto (rt_P, \underline{\tau}^P) & (nrt,\underline{\sigma}) \mapsto (rt_H,\underline{\tau}^H)
\end{align*}
These combine to a parametrized map
\[
F:\Sigma_{n,m,\ell} \xrightarrow{\;(F_H, \,F_P)\;} \T_n \times \T_m.
\]
The forgetful map is injective since by \cref{def:nestedtree} a nested phylogenetic tree is completely determined by the host tree, the parasite tree, and the parasite map $\ell$, which we fix. 
    The image of $F$ map corresponds exactly to those pairs of trees that are host-parasite compatible along $\ell$.
    
   There is also an adjoint map $i_\ell:(T_H,T_P)\to (T_H,T_P,\ell),$ defined on $im (F)$, where $i_\ell$ is restricted to host-parasite compatible trees. We see immediately that $i_\ell\circ F = id_{\Sigma_{n,m,\ell}}$.

\begin{lemma}
    $\Sigma_{n,m,\ell}$ is contractible.
\end{lemma}
\begin{proof}
    Every orthant and lower-dimensional face corresponds to a nested ranked topology,
    possibly with degeneracies or simultaneous speciations. Nevertheless, each is parametrized by its non-zero coordinates, and is contractible along the straight line homotopy shrinking $\underline{\sigma}\to \underline{0},$ which is continuous on shared faces.
\end{proof}

\begin{proposition}\label{prop:no3-cycle}
    The link of the origin in $\Sigma_{n,m,\ell}$ does not have 3-cycles, meaning that there are no three top-dimensional orthants ($\cong\Rp^{n+m-2}$) that pairwise share an $n+m-3$-dimensional facet.  
\end{proposition}

\begin{proof}
Suppose $o_1,o_2,o_3$ are three orthants pairwise sharing a facet.
Then all $\sigma_i$ coordinates are positive in the interiors of $o_1,o_2,o_3$, and at the boundary $\partial o_{ij}$ shared by $o_i$ and $o_j$, a unique coordinate $\sigma_{c_{ij}}$ vanishes. Since $o_i$ and $o_j$ are different orthants and thus correspond to different nested ranked topologies, that means that $\sigma_{c_{ij}}$ is resolved a different ranked topology in $o_i$ and $o_j$. Each codimension 1 boundary crossing through $\partial o_{ij}$ corresponds to an action $\rho_{ij}$ on the nested ranked tree topology specifying the orthant consisting of either:
\begin{enumerate}
    \item[(H1)] an NNI in the host tree;
    \item[(H2)] a rank change in the host tree;
    \item[(P1)] an NNI in the parasite tree;
    \item[(P2)] a rank change in the parasite tree; or
    \item[(HP)] a rank change between the parasite and host trees.
\end{enumerate}
Consider the cycle $\rho_{31}\circ\rho_{23}\circ\rho_{12}:o_1\to o_2 \to o_3 \to o_1$. If all three $\rho$'s act exclusively on host (resp. parasite) tree, i.e. factor through the appropriate forgetful map $\rho = F\circ\rho\circ F^{-1}$, then this would correspond to a cycle of length 3 in the link of the host (resp. parasite) tree space, which does not exist \cite{ultra}.

If two of the $\rho$'s, say $\rho_{12}$ and $ \rho_{23}$, act on the parasite tree (P1/P2), and $ \rho_{31}$ on the host tree (H1/H2), then in order to get back to the same nested ranked topology at the end of the cycle, $\rho_{12}$ and $ \rho_{23}$ have to act inversely on the parasite tree, in which case $o_1$ and $o_2$ share the boundary $\partial o_{23}$ as well as $\partial o_{12}$, in contradiction with the assumption. Analogously for two actions on the host tree.

Lastly, if one of the transitions is a rank change between the host and the parasite trees (HP), then exactly one other transition needs to be the inverse rank change for the nested ranked topology to return to itself along the cycle. Hence again two of the orthants would share two faces, in contradiction with the assumption.

Hence we conclude that there are no 3-cycles in the link of $\sigma$-space, analogous to the result in $\tau-$ and $\BHV$ space.
\end{proof}

\begin{theorem}[Gromov cubes]\label{cube condition}
Let $X$ be a cube complex. If it is the case that every three $k+2$-cubes sharing a common $k$-cube and pairwise sharing $k+1$-cubes, are contained in a common $k+3$-cube, then $X$ is CAT(0).
\end{theorem}

\begin{theorem}\label{thm:CAT(0)}
    $\sigma$-space $\Sigma_{n,m,\ell}$ satisfies Gromov's cube condition (\cref{cube condition}).
\end{theorem}
\begin{proof}
Let $k\leq n+m-4$ and suppose $o_1,o_2,o_3$ are $k+2$-cubes sharing a single common $k$-cube and pairwise sharing $k+1$-cubes. 
By \cref{prop:no3-cycle}, there is no such situation in which $o_1,o_2,o_3$ are top-dimensional, so $k< n+m-4$. We will now argue that in this case we can construct a $k+3$-cube filling the corner formed by the three $k+2$-cubes.
In this situation there are $k$ positive $\{\sigma_{c_i}\}_{i=1,\dots,k}$ dimensions common to $o_1,o_2,o_3,$ and a single additional $\sigma_{c_{i^c}}>0$ shared by each pair $\{o_j\}_{j \in i^c}$, for $i^c=\{1,2,3 \}\setminus i$, which vanishes on $o_i$.
As $\sigma_{c_{\{2,3\}}}>0$ on $o_2$ and $o_3$ as well as their shared boundary, the coordinate is resolved in the same way in both cubes, by which we mean that the nested ranked tree topologies corresponding to $o_2$ and $o_3$ are homeomorphic within the non-zero time interval between nodes $c_{\{2,3\}}$ and $c_{\{2,3\}}-1$. This nodes are at the same height in $o_1$. We can now construct a $k+3$-cube containing $o_1,o_2,o_3$ by taking points in $o_1$ and resolving $\sigma_{c_{\{2,3\}}}$ in the way it is resolved in $o_2$ and $o_3$.
\end{proof}

\begin{corollary}
    $\sigma$-space is CAT(0). It is a connected length space with unique geodesics.
\end{corollary}

As in the classical tree space literature (e.g. \cite{BHV}), we note that the non-positive curvature implies the existence of unique geodesics, which in turn can be used to compute Fr\'echet means of point clouds, as well as further geometric statistics. 

 \bibliographystyle{amsalpha}
  \bibliography{bib}

\end{document}